\documentclass[11pt]{article}
\usepackage{amsmath,amsthm,amsfonts,amssymb,amscd}
\usepackage[utf8]{inputenc}
\usepackage[T1]{fontenc}
\usepackage{lmodern}
\usepackage{graphics}
\usepackage{color}
\usepackage[pdftex]{graphicx}
\def\tto{\;{\lower 1pt\hbox{$\rightarrow$}}\kern -10pt
\hbox{\raise 2pt\hbox{$\rightarrow$}}\;}

\def\ra{\rangle}
\def\la{\langle}

\def\h{\hfill\Box}
\def\R{\mathbb R}

\def\N{\mathcal{N}}

\def\cone{\mbox{\rm cone}\,}

\def\dim{\mbox{\rm dim}\,}

\def\ker{\mbox{\rm ker}\,}

\def\cl*co{\mbox{\rm cl}^*\mbox{\rm co}\,}
\def\cl{\mbox{\rm cl}\,}

\def\h{\hfill\triangle}

\def\N{\mathbb N}
\def\hs7{\hspace*{7pt}}

\renewcommand{\theequation}{\thesection.\arabic{equation}}

\def\h{\hfill\Box}

\begin{document}
\newtheorem{Theorem}{Theorem}[section]
\newtheorem{Proposition}[Theorem]{Proposition}
\newtheorem{Remark}[Theorem]{Remark}
\newtheorem{Lemma}[Theorem]{Lemma}
\newtheorem{Corollary}[Theorem]{Corollary}
\newtheorem{Definition}[Theorem]{Definition}
\newtheorem{Example}[Theorem]{Example}
\newtheorem{Counterexample}[Theorem]{Counterexample}
\newtheorem{Fact}[Theorem]{Fact}
\renewcommand{\theequation}{\thesection.\arabic{equation}}
\normalsize
\def\proof{
\normalfont
\medskip
{\noindent\itshape Proof.\h
space*{6pt}\ignorespaces}}
\def\endproof{$\h$\vspace*{0.1in}}
\title{\bf Revisiting the Constant-Rank Constraint Qualification for Second-Order Cone Programs}
\date{}
\author{Nguyen Huy Chieu\footnote{Department of Mathematics, Vinh
		University, Vinh, Nghe  An, Viet Nam; email: chieunh@vinhuni.edu.vn},\, \ Nguyen Thi Quynh Trang\footnote{Department of Mathematics, Vinh
University, Vinh, Nghe  An, Viet Nam; email: quynhtrang@vinhuni.edu.vn} \, \  and\ \, \ Nguyen Thi Hai Yen \footnote{Corresponding author, Faculty of Mathematics and Information Technology, The University of Danang – University of Science and Education, Da Nang 550000,     Viet Nam; email: nthyen$\_$kt@ued.udn.vn. }}
\maketitle
\begin{quote}
{\small \noindent {\bf Abstract.}
 The constant rank constraint qualification (CRCQ) for second-order cone programs, introduced  by Andreani et al. in [{\it Math. Program.} 202 (2023), 473 - 513], shares some desirable properties with its classical nonlinear programming counterpart; specifically, it guarantees strong second-order necessary conditions for optimality, and is independent of the Robinson constraint qualification. However, unlike the classical version, this new CRCQ can fail  in the linear case, and it is unclear whether CRCQ implies  the metric subregularity constraint qualification (MSCQ). The aim of this paper is to examine the CRCQ for  second-order cone programs  in the linear setting. First, we  show  that the facial constant rank property, which is a key requirement for the validity of CRCQ, does not always hold in this context. Then, we derive a necessary and sufficient condition for a feasible point to satisfy this property.
 After that, we  establish an easily verifiable characterization of CRCQ. Finally, utilizing this characterization, we prove that CRCQ and MSCQ are equivalent.}

\end{quote}
\maketitle

\section{Introduction}
\setcounter{equation}{0}

The constant rank constraint qualification (CRCQ) introduced by Janin \cite{J84} plays important roles in nonlinear programming  \cite{GM15, ShuLu, L12,  MN15}. It can be utilized to establish stability results \cite{J84, GM15}, ensure algorithmic convergence \cite{andr08, andr07}, and derive  optimality conditions \cite{AES10, andr07, Chieu25, GM15,  ML16}.  Notably, CRCQ is strictly weaker than the linear independence constraint qualification (LICQ) but stronger than the metric subregularity constraint qualification (MSCQ); furthermore, it is independent of the Mangasarian-Fromovitz constraint qualification (MFCQ).  Unlike the latter, however, it is a strong second-order constraint qualification and holds at every feasible point of any linear program \cite{AES10}.

 The CRCQ has been extended and adapted to various significant classes of constrained optimization problems. For mathematical programs with equilibrium constraints (MPEC), Steffensen and Ulbrich \cite{SU10} introduced the MPEC-CRCQ to establish convergence for their proposed relaxation scheme. Similarly, Hoheisel et al. \cite{HKS12} presented the MPVC-CRCQ in the context of mathematical programs with vanishing constraints. More recently, Xu and Ye \cite{XY23} proposed the MPDC-CRCQ and the MPODC-CRCQ for mathematical programs with disjunctive and ortho-disjunctive constraints, respectively. These extensions are relatively natural and straightforward, as they are based on exploiting the polyhedral structure of the given sets.

The extension of CRCQ to nonpolyhedral conic programs is considerably more complex. While Zhang and Zhang \cite{zz19} made the initial attempt to address this, their approach was subsequently shown to be invalid \cite{AF21}. Motivated by this and its potential for analyzing the convergence and reliability of optimization algorithms, Andreani and coworkers introduced several variants of CRCQ for second-order cone programs (SOCP) and semidefinite programs (SDP) \cite{AHMR23,AHMRSS22, andreani23c}.  In particular, Andreani et al. \cite{andreani23c} employed cone reducibility and facial structure of the underlying cone to geometrically characterize the classical relaxed CRCQ, and  utilized this to extend  the notion to nonlinear SOCP and SDP, establishing what is now known as CRCQ in those contexts. Such an approach was later applied to define CRCQ in reducible conic programming \cite{andreani23b}. The resulting CRCQ shares several desirable properties with its classical nonlinear programming counterpart; specifically, it guarantees strong second-order necessary conditions for optimality and remains independent of the Robinson constraint qualification.  However, unlike the classical version, this new CRCQ can fail  in the linear case, and it is unclear whether CRCQ implies  the metric subregularity constraint qualification (MSCQ).

Because CRCQ is not universally satisfied at feasible points of linear nonpolyhedral cone programs, characterizing the specific points where the condition holds is essential. Such a characterization enables us to identify the particular classes of nonpolyhedral cone programs  where the CRCQ remains a valid tool for analysis.
      Furthermore, given that second-order variational analysis is well-established for nonpolyhedral cone programs under MSCQ \cite{BGM19, GM19, Hang20b, Mor24}, and considering that the link between CRCQ and MSCQ remains poorly understood in this framework, investigating their relationship is of significant practical and theoretical value.
      
    In this paper, we examines the CRCQ in linear nonpolyhedral second-order cone settings, focusing on the two problems identified above. First, we  show  that the facial constant rank property, which is a key requirement for the validity of CRCQ, does not always hold in this context. Then, we derive a necessary and sufficient condition for a feasible point to satisfy this property.  After that, we  establish an easily verifiable characterization of CRCQ. Finally, utilizing this characterization, we prove that CRCQ and MSCQ are equivalent.
   
 The paper is organized as follows. Section 2 recalls essential preliminaries, including the notions of cones, faces,  reducibility, and constraint qualifications within SOCPs. Section 3 investigates the facial constant rank property, where we derive a necessary and sufficient condition for its satisfaction at a feasible point and demonstrate that the property is not locally preserved. The characterization of CRCQ is established in Section 4. Subsequently, Section 5 explores the relationship between CRCQ and MSCQ. Finally, Section 6 summarizes our findings and outlines directions for future research.

\vspace{12pt}

\section{Preliminaries}
\setcounter{equation}{0}
This section recalls  some  concepts, notations and results from variational analysis and optimization (see, e.g., \cite{andreani23c,BS00, Mord06, rock70, rw}) that are needed for  our subsequent analysis.

Unless otherwise stated, $\mathbb{R}^n$ denotes the $n$-dimensional Euclidean space equipped with the standard inner product $\langle \cdot, \cdot \rangle$ and the induced norm $\|\cdot\|$.  We denote the nonnegative and positive orthants by $\mathbb{R}^n_+$ and $\mathbb{R}^n_{++}$, respectively.The set of all $m \times n$ real matrices is denoted by $\mathbb{R}^{m \times n}$. Following \cite{Luenb69}, vectors $x \in \mathbb{R}^n$ are represented as $n$-tuples $(x_1, \dots, x_n)$ but are treated as column vectors in all matrix operations. The transpose of a matrix $A$ is denoted by $A^*$.
For a real-valued function $f: \mathbb{R}^n \to \mathbb{R}$, we denote its gradient vector and Hessian matrix at $x$ by $\nabla f(x)$ and $\nabla^2 f(x)$. For a vector-valued function $g: \mathbb{R}^n \to \mathbb{R}^m$ with $m>1$, $\nabla g(x)$ represents both the Jacobian and the Fréchet derivative.
For a set $\Omega \subseteq \mathbb{R}^n$, we denote its interior, closure,  and boundary by $\text{int}(\Omega)$, $\text{cl}(\Omega)$,  and $\text{bd}(\Omega)$, respectively. The polar cone of $\Omega$  is $\Omega^* := \{z \in \mathbb{R}^n \mid \langle z, x \rangle \leq 0 \text{ for all } x \in \Omega\}$. We denote the smallest cone containing $\Omega$, the spanned subspace, and its orthogonal complement by $\text{cone}(\Omega)$, $\text{span}(\Omega)$, and $\Omega^\perp$, respectively. For a linear mapping $A$, the symbols $\text{Im}(A)$ and $\text{Ker}(A)$ are  its image and kernel, respectively. The Euclidean projector of $x\in \R^n$ onto $\Omega$  is the set  $\Pi_\Omega(x)$  defined as
$$\Pi_\Omega(x):=\{u\in \Omega\ |\ \|x-u\|=\text{dist}(x; \Omega)\},$$
where  $\text{dist}(x; \Omega):=\inf\limits_{u\in \Omega}\|x-u\|$ is the Euclidean distance from $x$ to $\Omega$.

Let  $C$ be a nonempty convex set in $\R^n.$ Recall that:
\begin{itemize}
\item A nonempty convex subset $F$ of $C$  is said to be a {\it face} of $C$, denoted by $F\unlhd C$,  if  $z, w\in F$ whenever  $z, w\in C$ with $\alpha z +(1-\alpha)w\in F$ for some $\alpha\in (0, 1)$; see \cite[Section 18]{rock70}.
\item Given a set $D\subset C$, the {\it minimal face} associated with $D$, denoted by $F_{\rm min}(D)$, is defined as the smallest face of $C$ that contains
    $D$.
\item The {\it relative interior} of $C$ is the set $\text{\rm ri}(C)$ given by
\[\text{\rm ri}(C):=\left\{x\in \text{\rm aff}(C)|\; \exists r>0 \ \mbox{such that}\  \mathbb{B}_r(x)\cap \text{\rm aff}(C)\subset C\right\},\]
where $\mathbb{B}_r(x):=\{u\in \R^n\ |\ \|u-x\|\leq r\}$ and $\text{\rm aff}(C)$ denotes the {\it affine hull} of $C$.
\item The set of {\it feasible directions} of $C$ at $\bar x\in C$ is the set $\text{\rm dir}(\bar x,C)$ defined by
$$\text{\rm dir}(\bar x,C):=\{d\in \R^n\ |\ \bar x+td\in C\quad \mbox{for some}\ t>0\}.$$
\item  $C$ is called a {\it nice cone} if it is a closed cone and  $C^*+E^{\bot}$ is closed for every $E\unlhd C$; see \cite[p. 396]{Pataki07}.
\end{itemize}

It is well-known  that both the $n$-dimensional Euclidean space $\mathbb{R}^n$ and the second-order cones (or Lorentz cones) in $\mathbb{R}^m$, defined as$$\mathcal{Q}_{m} := \{ (v_0, v_r) \in \mathbb{R} \times \mathbb{R}^{m-1} \;|\; v_0 \geq \|v_r\| \}$$for $m \in \mathbb{N}^*:=\{1,2,...\}$, belong to the class of nice cones \cite[p. 399]{Pataki07}.
 The second-order cone $\mathcal{Q}_{m}$ has  infinitely many faces for $m > 2$. These are categorized into three distinct types: the vertex $\{0\}$, the entire cone $\mathcal{Q}_m$, and the extreme rays originating at the vertex and passing through any point on the non-zero boundary $v \in \text{bd}^+(\mathcal{Q}_m) := \text{bd}(\mathcal{Q}_m) \setminus \{0\}$; see \cite[p. 488]{andreani23c}. The faces of the non-negative orthant $\mathbb{R}_+$ are only the origin and the cone itself.

Let  $\Omega\subset\mathbb{R}^n$ be  a non-empty set locally closed around $\bar x\in \Omega$, that is, $\Omega\cap U$ is closed for some neighborhood $U$ of $\bar x$. Recall \cite{Mord06,Mor24,rw} that:
\begin{itemize}		
\item The (Bouligand-Severi) {\it tangent cone} to the set $\Omega$ at   $\bar{x}\in \Omega$  is the set $T_\Omega(\bar x)$ defined by
    \begin{equation*}
  	T_\Omega(\bar x):=\big\{d\in\R^n | \, \exists  t_k \downarrow 0, \  d_k\rightarrow d\ \mbox{ with }\  \bar x+t_kd_k\in\Omega \quad \forall    k\in \N^*\big\}.
  \end{equation*}
\item  The (Mordukhovich) {\it limiting normal cone} to the set $\Omega$ at $\bar x$ is the set $N_{\Omega}(\bar x)$ given by
\begin{equation*}
N_{\Omega}(\bar x):=\left\{v\in \R^n\ |\ \exists x_k\to \bar x,  v_k\to v \ \mbox{with}\   v_k\in \cone\big(x_k-\Pi_{\Omega}(x_k)\big)\quad
 \forall k\in \N^*\right\}.
\end{equation*}
If $\bar x\not\in \Omega,$ put $T_\Omega(\bar x):=\emptyset$ and  $N_{\Omega}(\bar x):=\emptyset$ by convention.
\end{itemize}

Note  that   $\bar{x} \in \operatorname{int}(\Omega)$ iff  $N_\Omega(\bar{x}) = \{0\}$;  see \cite[Corollary 2.24]{Mord06} or \cite[Exercise 6.19]{rw}.
For a closed convex set $\Omega$ containing $\bar x$, the tangent cone $T_\Omega(\bar x)$ and the normal cone $N_\Omega(\bar x)$ can be expressed simply as follows:
$$T_\Omega(\bar x) = \text{cl} \left( \bigcup_{\lambda > 0} \lambda(\Omega - \bar x) \right),$$
and $$N_\Omega(\bar x) =[T_\Omega(\bar x)]^*= \{ v \in \mathbb{R}^n \mid \langle v, x - \bar x \rangle \leq 0, \forall x \in \Omega \}.$$
Furthermore, for second-order cones, it is not difficult to see that
\begin{eqnarray}\label{t2ocone}
T_{\mathcal{Q}_m}(x)=\begin{cases}
	\mathbb{R}^m & \text{ if } x\in \text{int} (\mathcal{Q}_m),\\
	\mathcal{Q}_m & \text{ if } x=0,\\
	\{x'\in\mathbb{R}^m\;|\; \la \widetilde{x}, x'\ra\leqslant0\}& \text{ if } x\in \text{bd}^+(\mathcal{Q}_m),
\end{cases}
\end{eqnarray}
and
\begin{eqnarray}\label{n2ocone}
	N_{\mathcal{Q}_m}(x)=\begin{cases}
		\{0\} & \text{ if } x\in \text{int} (\mathcal{Q}_m),\\
		-\mathcal{Q}_m & \text{ if } x=0,\\
		\mathbb{R}_+\widetilde{x}& \text{ if } x\in \text{bd}^+(\mathcal{Q}_m),
	\end{cases}
\end{eqnarray}
where $\widetilde{x}=(-x_0, x_r)$ if $x=(x_0, x_r)\in\mathbb{R}\times\mathbb{R}^{m-1}$; see, e.g., \cite[p.13]{Hang20}.

Recall from \cite[Definition 3.135]{BS00} that a closed convex cone $\mathcal{K}\subset \mathbb{R}^m$ is said to be 	{\it $C^2$-reducible} to a closed convex pointed cone $\mathcal{C}\subset \R^\ell$ at a point $\bar y\in \mathcal{K}$ if there exists a neighborhood $\mathcal{N}$ of $\bar y$ and a twice continuously	differentiable reduction mapping $\Xi: \mathcal{N} \to \mathbb{R}^\ell$  such that  $\Xi(\bar y)=0$,  $\nabla \Xi(\bar y)$ is surjective, and
$\mathcal{K}\cap \mathcal{N}=\{y\in \mathcal{N}\,| \,\Xi(y)\in \mathcal{C}\}.$

When $\mathcal{K} = \mathcal{Q}_m$, a natural reduction mapping $\Xi$ at $\bar y \in \mathcal{Q}_m$ can be chosen as follows:
\begin{itemize}
\item If $\bar y \in \text{int}(\mathcal{Q}_m)$,  let  $\Xi(y) = 0$ with $\mathcal{C} = \{0\}$.
\item If $\bar y = 0$, let  $\Xi(y) = y$ with $\mathcal{C} = \mathcal{Q}_m$.
\item If $\bar y \in \text{bd}^+(\mathcal{Q}_m)$, let  $\Xi(y) := y_0 - \|y_r\|$ with $y=(y_0, y_r)$ and $\mathcal{C} = \mathbb{R}_+$.
\end{itemize}
We refer to this construction as the {\it natural reduction}  for $\mathcal{Q}_m$.

Consider the following second-order cone program (SOCP):
\begin{equation}\begin{array}{rl}  &  \quad \min\limits_{x\in\mathbb{R}^n}\quad \ \,\quad \ f(x) \\
	& \mbox{subject to}\quad \ \,   g(x)\in \mathcal{Q}_m,
\end{array} \label{SOCP0}\end{equation}
where  both $f: \mathbb{R}^n\to \mathbb{R}$ and $g: \mathbb{R}^n\to \R^m$  are  twice continuously differentiable  around  $\bar x\in \Omega:=\{x\in \R^n\ |\ g(x)\in \mathcal{Q}_m\}$.
 If $\mathcal{Q}_m$ is reducible to $\mathcal{C}$ at $g(\bar x)$ by the reduction function $\Xi$, define $\mathcal{G}:= \Xi\circ g$,  then the reduced
constraint $\mathcal{G}(x)\in \mathcal{C}$ is  equivalent to the original constraint $g(x)\in \mathcal{Q}_m$ in a
neighborhood of $\bar{x}$ with $\mathcal{G}(\bar{x})=0$.  Furthermore, the reduced problem formulated as
\begin{equation}\begin{array}{rl}  &  \quad \min\limits_{x\in\mathbb{R}^n}\quad \ \,\quad \ f(x) \\
	& \mbox{subject to}\quad \ \,   \mathcal{G}(x)\in \mathcal{C},
\end{array} \label{Red-NCP}\end{equation}
is locally equivalent to the original problem \eqref{SOCP0}
around  $\bar{x}$.

To date, the following constraint qualifications still play central roles in both the theoretical and practical study of SOCPs \cite{AG03, BGM19, BS00, Hang20, Hang20b, Mor24}.
		\begin{itemize}
			\item The {\it  nondegeneracy condition} is valid  at $\bar x$ iff 	\begin{equation*}\label{nondegeracy}
				\text{Im}(\nabla g(\bar x))+\text{lin}\left(T_{\mathcal{Q}_m}(g(\bar x))\right) =\R^m,		
			\end{equation*}
		where $\text{lin}\left(T_{\mathcal{Q}_m}(g(\bar x))\right) :=T_{\mathcal{Q}_m}\big(g(\bar x)\big)\cap \left[-T_{\mathcal{Q}_m}(g(\bar x))\right]$.
		\item The {\it Robinson  constraint qualification} (RCQ) is satisfied at $\bar{x}$ iff
		\begin{equation*}\label{2178}
			0\in\text{int} \left( g(\bar x)+ \text{Im}(\nabla g(\bar x))-\mathcal{Q}_m\right).
		\end{equation*}
	\item The {\it metric subregularity constraint qualification} (MSCQ) holds  at $\bar x\in \Omega$ iff  there exist real numbers $\kappa, r>0$ such that
	\[\text{dist}(x, \Omega)\leqslant \kappa \text{dist}\big(g(x), \mathcal{Q}_m\big) \text{ for all } x\in \mathbb{B}_r(\bar x),\]	
		where  the distance from $y\in \R^m$ to $\mathcal{Q}_m$ can be calculated by 
\begin{equation}\label{h312}\text{\rm dist}\left(y, \;\mathcal{Q}_m \right)=\begin{cases}
		0 & \text{\rm if } y\in \mathcal{Q}_m,\\
		\|y\| & \text{\rm if } y\in -\mathcal{Q}_m,\\
		\frac{\sqrt{2}}{2}\left(\|y_r\|-y_0 \right) & \text{\rm if } y\not\in \mathcal{Q}_m\cup \left( -\mathcal{Q}_m\right);
	\end{cases} \end{equation}
see \cite[p. 37]{Hang20} or \cite[Section 3.3.4]{Bauschke90}. 
\end{itemize}

If $\bar x$ is a locally optimal solution of (SOCP) at which RCQ holds, then the set of Lagrange   multipliers  $\Lambda(\bar{x}) :=\{\lambda\in \mathcal{Q}_m^*\ |\ \nabla f(\bar x) + \nabla g(\bar x)^*\lambda = 0, \,  \la g(\bar x), \lambda\ra =0
 \}$ is  non-empty and compact \cite[Theorem 3.9]{BS00}. By  \cite[Proposition  2.97]{BS00}, the validity of RCQ at $\bar x$ is equivalent to
 that  $\text{Im}(\nabla g(\bar x))-T_{\mathcal{Q}_m}\big(g(\bar x)\big) =\R^m.$
  Consequently, the nondegeneracy condition is  stronger than RCQ.  By \cite[Theorem 2.87]{BS00}, RCQ in turn implies MSCQ. Moreover, given the reducibility of $\mathcal{Q}_m$ to $\mathcal{C}$, nondegeneracy at $\bar{x}$ is equivalent to the surjectivity of $\nabla \mathcal{G}(\bar{x})$, which ensures $\Lambda(\bar{x})$ is a singleton when  $\bar x$ is a locally optimal solution  \cite[Proposition 4.75]{BS00}.
 Despite being a weak constraint qualification, the MSCQ  remains essential  for developing second-order variational analysis in second-order cone programming  \cite{Hang20, Hang20b}.

The {\it linearized tangent cone} (or simply {\it linearized/linearization cone}) to $\Omega$ at $\bar{x}\in \Omega$ is the set $L_\Omega(\bar x)$ given by
\begin{equation*}\label{tcone}
	L_\Omega(\bar x):=\big\{d\in\R^n | \, \nabla g(\bar x)d\in T_{\mathcal{Q}_m}(g(\bar x))\},
\end{equation*}
which  provides a commonly used outer approximation of the tangent cone $T_\Omega(\bar x)$. 
This cone  can be represented as
   \begin{equation*}\label{linconefornsocp}
 	L_\Omega(\bar x)=\begin{cases}
   \mathbb{R}^n  &\text{ if } g(\bar x)\in \text{int}(\mathcal{Q}_m),\\
 \left\{ d\in\mathbb{R}^n\ |\  \nabla g(\bar x)d\in \mathcal{Q}_{m}\right\}\quad  &\text{ if } g(\bar x)=0,\\
   \left\{ d\in  \mathbb{R}^n\ |\  \nabla\phi(\bar x)d\in \mathbb{R}_+ \right\} &\text{ if } g(\bar x)\in \text{bd}^+(\mathcal{Q}_m),
      \end{cases}
 \end{equation*}
where $\phi(x):=g_0(x)-\|g_r(x)\|$; see \cite[Lemma 25]{bonnans05}. According to \cite{andreani23b}, we have
$$L_{\Omega}(\bar{x}) = \left\{ d \in \mathbb{R}^n \mid \nabla \mathcal{G}(\bar{x})d \in \mathcal{C} \right\}, \,\  H(\bar x) = \nabla \mathcal{G}(\bar{x})^* \mathcal{C}^*\,\  \mbox{and}\, \ L_{\Omega}(\bar{x})^* = \text{cl} \left( H(\bar{x}) \right),$$
where \begin{equation}\label{setHx}H(\bar x):= \nabla g(\bar x)^*\left[N_{\mathcal{Q}_m}\big(g(\bar x)\big)\right].\end{equation}
Furthermore, if MSCQ holds at $\bar x\in \Omega$ then
$$T_\Omega(\bar x)= L_\Omega(\bar x)\quad \mbox{and}\quad N_\Omega(\bar x)=H(\bar x).$$

Recently, Andreani et al. \cite{andreani23b, andreani23c} extended the classical constant rank constraint qualification from nonlinear programming to broader settings, including second-order cone programming, semidefinite programming, and reducible conic programming. This generalization provides a constant rank constraint qualification (CRCQ) for conic programs that  ensures the fulfillment of the strong second-order necessary conditions for optimality, while remaining independent of the Robinson constraint qualification. In the specific context of second-order cone programming (SOCP), their CRCQ is formulated as follows:

Given any point $\bar x\in\Omega:=\{x\in \R^n\ |\ g(x)\in \mathcal{Q}_m\}$
such that $\mathcal{Q}_m$ is reducible to the cone $\mathcal{C}\subset \mathbb{R}^\ell$ at $g(\bar x)$ by the natural reduction mapping $\Xi$.
\begin{itemize}
\item The {\it facial constant rank (FCR) property} holds at $\bar x$ if there exists a neighborhood $\mathcal{V}$ of $\bar x$ such that, for every $F \unlhd \mathcal{C}$, the dimension of $\nabla \mathcal{G}(x)^*(F^{\perp})$ remains constant for every $x\in\mathcal{V}$.
 \item The {\it constant rank constraint qualification} (CRCQ) holds at $\bar x$ if the FCR property is satisfied at $\bar x$ and, additionally, the set $H(\bar x)$ defined in $\eqref{setHx}$ is closed.	
\end{itemize}

If the FCR property holds at $\bar{x} \in \Omega$, then  $T_{\Omega}(\bar{x}) = L_{\Omega}(\bar{x})$; see \cite[Theorem 3]{andreani23c}. However, this  property alone does not guarantee that the set $H(\bar{x})$ is closed; consequently, it is insufficient to ensure the validity of MSCQ; see \cite[Example 1]{andreani23c}.
Note that, in the nonpolyhedral conic setting, it remains an open question whether the CRCQ implies the MSCQ.
\section{The Facial Constant Rank Property for Affine Second-order Cone Constraints}
\setcounter{equation}{0}
This section investigates the facial constant rank (FCR) property for  affine second-order cone constraints of the form:
\begin{equation}\label{SOCP} g(x):= Ax+b\in \mathcal{Q}_m, \end{equation}
where	$A \in \mathbb{R}^{m \times n}$ and $b \in \mathbb{R}^m$.

Let $\bar{x} \in \mathbb{R}^n$ such that $g(\bar{x}) \in \mathcal{Q}_m$. By utilizing the natural reduction mapping $\Xi$ for $\mathcal{Q}_m$ at $g(\bar x)$, which is defined on some neighborhood $\mathcal{N}$ of $g(\bar{x})$,   the constraint $g(x) \in \mathcal{Q}_m$ is locally represented by $\mathcal{G}(x) \in \mathcal{C}$  in the neighborhood $\mathcal{U} := \{x \mid g(x) \in \mathcal{N}\}$. Furthermore, with the reduced mapping $\mathcal{G} := \Xi \circ g$, the derivative $\nabla \mathcal{G}(x)$ is obtained via the chain rule as follows:
\begin{equation}\label{dhG}
\nabla \mathcal{G}(x)(u) = \nabla \Xi\big(g(x)\big)\circ \nabla g(x)(u) = \nabla \Xi\big(g(x)\big)(Au)\quad \mbox{for all}\ x\in \mathcal{U},\  u\in \R^n.
\end{equation}

Contrary to \cite[Remark 2]{andreani23c}, which suggests the FCR property holds at every feasible point for all affine second-order cone constraints, the following counterexample demonstrates that this property is not always satisfied.
\begin{Counterexample}\label{ce1}{\rm
Let us consider the constraint \eqref{SOCP} with  $m=n=3$,  $b=0\in \R^3$, and  $Ax=(x_1, x_1, x_3)$  for  $x=(x_1, x_2, x_3)\in\mathbb{R}^3$:
 $$g(x):=Ax\in \mathcal{Q}_3.$$
 We see that   $\bar{x} = (1, 0, 0)$ is a feasible point, as it satisfies $\bar{x} \in \Omega := \{x \in \mathbb{R}^3 \;|\; g(x) \in \mathcal{Q}_3\}$.
 Since  $\bar{y}:= g(\bar x)= (1, 1, 0) \in {\rm bd}^+ (\mathcal{Q}_3)$, the second-order cone $\mathcal{Q}_3$ in the three-dimensional space is reducible at $\bar y$ to the cone $\mathcal{C}=\mathbb{R}_+$ in some neighborhood $\mathcal{N}$ of $\bar y$ by the reduction mapping $\Xi: \mathcal{N}\to \mathbb{R}$ given by $\Xi(y)=y_0-\sqrt{y_1^2+y_2^2}$ for every $y=(y_0, y_1, y_2)\in\mathcal{N}$. Note that, in this case,  $\mathcal{N}$ can be chosen as $\mathcal{N}:=\mathbb{R}\times\left(\mathbb{R}^2\setminus\{(0,0)\}\right)$, and
 \begin{equation}\label{eq1}
	\nabla\Xi(y)=\left(1, -\frac{y_1}{\sqrt{y_1^2+y_2^2}}, -\frac{y_2}{\sqrt{y_1^2+y_2^2}} \right)\;  \mbox{ for  every}\; y:=(y_0, y_1, y_2)\in \mathcal{N}.
\end{equation}
Let $\mathcal{G}(x):=\Xi(g(x))=\Xi(Ax)$. Then, combining the chain rule for derivatives with \eqref{eq1}, we get
$$\nabla \mathcal{G}(x)(u)=  \nabla \Xi(Ax)(Au)=\left(1-\frac{x_1}{\sqrt{x_1^2+x_3^2}}\right)u_1 -\frac{x_3}{\sqrt{x_1^2+x_3^2}}u_3.$$
So, choosing $F=\{0\}\subset\R$, we see that   $F \unlhd \mathcal{C}$,  $F^{\perp}=\mathbb{R}$, and hence,
\begin{eqnarray*}
	\nabla \mathcal{G}(x)^*(F^{\perp})&=&\left(A^*\circ \nabla \Xi(Ax)^* \right) \left( \mathbb{R}\right)\\
	&=& {\rm span}\left\lbrace \left(1-\frac{x_1}{\sqrt{x_1^2+x_3^2}},\; 0,\; -\frac{x_3}{\sqrt{x_1^2+x_3^2}} \right) \right\rbrace,
\end{eqnarray*}
for all $x\in \R^3$ with $Ax\in \mathcal{N},$ and  $u\in \R^3$.
Therefore,
\begin{eqnarray*}
	\text{ dim } \nabla\mathcal{G}(x)^*(F^{\perp})
	&=&  \text{ dim }\left( {\rm span}\left\lbrace \left(1-\frac{x_1}{\sqrt{x_1^2+x_3^2}},\; 0,\; -\frac{x_3}{\sqrt{x_1^2+x_3^2}} \right) \right\rbrace\right) \\
	&=& \begin{cases}
		0 \; & \text{ if } \; x_3=0 \text{ and } x_1>0,\\
		1\; & \text{ otherwise},
	\end{cases}
\end{eqnarray*}
where $x=(x_1, x_2, x_3)\in\mathbb{R}^3$ with $x_1^2+x_3^2>0$.
In particular, we observe that $\dim \nabla \mathcal{G}(\bar{x})^*(F^{\perp}) = 0$ and $\dim \nabla \mathcal{G}(x^n)^*(F^{\perp}) = 1$ for the sequence $x^n := (1, 0, \frac{1}{n})$, which converges to $\bar{x}$ as $n \to \infty$. This demonstrates that $\bar{x}$ fails to satisfy the FCR property, notwithstanding the affinity of the constraint function $g$.
}
\end{Counterexample}

Counterexample \ref{ce1} naturally raises the question of which feasible points of \eqref{SOCP} actually satisfy the FCR property. The following theorem provides a characterization of these points.
 \begin{Theorem}\label{cy209}
   Let $\bar{x}$ be a feasible point of \eqref{SOCP}. Then, the FCR property holds  at $\bar{x}$ if and only if one of the following conditions is satisfied:
  \begin{itemize}
  \item[$(i)$] $g(\bar{x}) = 0$;
 \item[$(ii)$] $g(\bar{x}) \in {\rm int}(\mathcal{Q}_m)$;
    \item[$(iii)$] $g(\bar x)\in {\rm bd}^+(\mathcal{Q}_m)$ and the nondegeneracy condition holds at $\bar{x}$;
 \item[$(iv)$] $g(\bar x)\in {\rm bd}^+(\mathcal{Q}_m)$ and the reduced mapping $\mathcal{G}$ vanishes  on a neighborhood of $\bar{x}$.

  \end{itemize}
 \end{Theorem}
\noindent {\it Proof}. Suppose first that one of the following conditions $(i)$ through $(iv)$ is satisfied. We aim to show that the FCR property holds at $\bar{x}$.

{\it Case 1.1}:  $g(\bar x)=0$. Then, the reduction mapping $\Xi$ is the identity on $\mathbb{R}^m$, and the reduced cone $\mathcal{C}$ coincides with $\mathcal{Q}_m$. Hence, $\mathcal{G}$ is an affine mapping, and $\nabla \mathcal{G}(x)$ does not depend on $x$. Consequently,  for every $F \unlhd \mathcal{C}$, the dimension of  $\nabla g(x)^*(F^{\perp})$ is constant in a neighborhood of $\bar{x}$.  This shows that the FCR property holds at $\bar x$.

{\it Case 1.2}:  $g(\bar{x}) \in {\rm int}(\mathcal{Q}_m)$.  In this case,  the reduction mapping $\Xi$ is the zero map, and the reduced cone $\mathcal{C}$ is trivial, i.e., $\mathcal{C} = \{0\}$. Therefore, $\mathcal{G}$ is  a zero mapping, and $\nabla \mathcal{G}(x)$ does not depend on $x$. It follows that  for every $F \unlhd \mathcal{C}$, the dimension of  $\nabla g(x)^*(F^{\perp})$ is constant in a neighborhood of $\bar{x}$.  So, the FCR property holds at $\bar x$.

{\it Case 1.3:} $g(\bar x)\in {\rm bd}^+(\mathcal{Q}_m)$  and the nondegeneracy is valid  at $\bar{x}$. Then, given that nondegeneracy is a stronger condition than CRCQ \cite[p.~492]{andreani23c},  the FCR property holds at $\bar{x}$ by definition.

{\it Case 1.4:}  $g(\bar x)\in {\rm bd}^+(\mathcal{Q}_m)$  and the reduced mapping $\mathcal{G}$ vanishes on a neighborhood of $\bar{x}$. Then, $\nabla \mathcal{G}(x) = 0$ for all $x$ in a neighborhood of $\bar{x}$. Consequently, for every face $F \unlhd \mathcal{C}$, the dimension of  $\nabla\mathcal{G}(x)^*(F^{\perp})$ is zero for every $x$ sufficiently close to $\bar{x}$. This ensures that the FCR property is satisfied at $\bar{x}$.

Conversely, suppose that the FCR property holds at $\bar{x}$ and that conditions $(i)$ through $(iii)$ are not satisfied. We need to  show  that condition $(iv)$ must hold. Since conditions $(i)$ through $(iii)$ fail to hold, it follows that $(\bar y_0, \bar y_r)=\bar y:=g(\bar{x}) \in \text{bd}^+(\mathcal{Q}_m)$ and the nondegeneracy condition is not satisfied at $\bar{x}$. By the FCR property at $\bar{x}$, there exists a neighborhood $\mathcal{V}$ of $\bar{x}$ such that for every face $F \unlhd \mathcal{C}$, the dimension of $\nabla \mathcal{G}(x)^*(F^{\perp})$ is constant for all $x \in \mathcal{V}$. Here, $\mathcal{G} := \Xi \circ g$ is the reduced mapping with $\Xi(y) := y_0 - \|y_r\|$ defined on a neighborhood $\mathcal{N}$  of $g(\bar{x})$, and the reduced cone is $\mathcal{C} = \mathbb{R}_+$.
 Furthermore,  by shrinking $\mathcal{V}$ and $\mathcal{N}$, we may assume $\mathcal{V}$ is a convex neighborhood of $\bar{x}$ contained in $\mathcal{U} := \{x\in\mathbb{R}^n|\; g(x) \in \mathcal{N}\}$, and  $y_r \neq 0$ for all $y = (y_0, y_r) \in \mathcal{N}$. This yields
\[\nabla\Xi(y)=\left( 1, -\frac{y_r}{\|y_r\|}\right)  \ \mbox{for all}\   y =(y_0, y_r)\in \mathcal{N}.\]
Set $F=\{0\}\subset \R$.
 Then, $F  \unlhd  \mathcal{C}$ and
 $F^{\perp}=\mathbb{R}$.
 Since the nondegeneracy condition fails at $\bar x$,  the linear mapping  $\nabla \mathcal{G}(\bar x): \R^n\to \R$ is not  surjective. This implies that $\nabla\mathcal{G}(\bar x)=0$, and thus $\nabla\mathcal{G}(\bar x)^*=0.$
  Consequently, we find that $$\dim \nabla \mathcal{G}(\bar{x})^*(F^{\perp}) = \dim \nabla \mathcal{G}(\bar{x})^*(\mathbb{R}) = 0.$$ Since  $\dim \nabla \mathcal{G}(x)^*(F^{\perp})$ is constant on the neighborhood $\mathcal{V}$, it follows that
  $$\dim \nabla \mathcal{G}(x)^*(\mathbb{R}) = \dim \nabla\mathcal{G}(x)^*(F^{\perp}) = 0,$$ for all $x \in \mathcal{V}$.  This implies that  $\nabla\mathcal{G}(x) = 0$ for all $x \in \mathcal{V}$. So, given the convexity of $\mathcal{V}$ and the fact that $\mathcal{G}(\bar{x}) = 0$, we conclude that $\mathcal{G}(x) = 0$ for all $x \in \mathcal{V}$.
\endproof

We now turn to the local preservation of the FCR property. By local preservation, we mean that if the property holds at a feasible point $\bar{x}$, it  also holds at every feasible point near $\bar{x}$.  The following proposition establishes that the FCR property does not, in general, possess this stability.

\begin{Proposition} The facial constant rank property for affine second-order cone constraints of the form \eqref{SOCP} is not locally preserved. \end{Proposition}
\noindent {\it Proof}. Consider the affine second-order cone constraint  as defined in Counterexample \ref{ce1}:
 $$g(x) := Ax \in \mathcal{Q}_3,$$
 where  $Ax = (x_1, x_1, x_3)$ for $x=(x_1,x_2,x_3)\in \R^3$.  Let  $\bar x= (0, 0, 0)$ and  $\Omega:=\{x\in\mathbb{R}^3\;| \; g(x)\in \mathcal{Q}_3\}$. We see that
 $$\Omega=\left\{x\in \mathbb{R}^3\;|\; Ax\in\mathcal{Q}_3\right\}=\left\{(x_1, x_2, x_3)\in \mathbb{R}^3\;\big|\;x_1\geqslant\sqrt{x_1^2+x_3^2}\right\}=\mathbb{R}_{+}\times\mathbb{R}\times\{0\}.$$
 Since $g(\bar x)=0$, by Theorem $\ref{cy209}$, the FCR property holds at $\bar x$.
Put $x^{n}=\left(\frac{1}{n}, 0, 0 \right)$ for $n\in \mathbb{N}^*$. Then $\{x^{n}\} \subset \Omega$ converges to $\bar{x}$.
As $g(x^{n})=\left(\frac{1}{n}, \frac{1}{n}, 0 \right)\in \text{bd}^+(\mathcal{Q}_3)$ for every $n\in \N^*$, the second-order cone $\mathcal{Q}_3$ in the three-dimensional space $\R^3$  is reducible at $g(x^{n})$ to the cone $\mathcal{C}=\mathbb{R}_+$  by the reduction mapping $\Xi: \mathcal{N}\to \mathbb{R}$ given by $\Xi(y)=y_0-\sqrt{y_1^2+y_2^2}$ for every $y=(y_0, y_1, y_2)\in\mathcal{N},$
where $\mathcal{N}=\{y=(y_0,y_1,y_2)\ | \ y_1^2+y_2^2\not=0\}$ is a neighborhood of $g(x^n)$ for every $n\in \N^*$.
Taking $F=\{0\}\unlhd\mathcal{C}$ and proceeding as in Counterexample \ref{ce1}, we find that:
\begin{eqnarray*}
	\text{ dim } \nabla \mathcal{G}(x)^*(F^{\perp})
	&=& \begin{cases}
		0 \; & \text{ if } \; x_3=0 \text{ and } x_1>0,\\
		1\; & \text{ otherwise},
	\end{cases}
\end{eqnarray*}
where $x=(x_1, x_2, x_3)\in\mathbb{R}^3$ with $x_1^2+x_3^2>0$. Hence,  $\text{ dim } \nabla \mathcal{G}(x^{n})^*(F^{\perp})=0$ for every $n\in \mathbb{N}^*$.
On the other hand, defining $x^{n, k} := \left(\frac{1}{n}, 0, \frac{1}{k} \right)$, we have $\dim \nabla \mathcal{G}(x^{n, k})^*(F^{\perp}) = 1$ for all $n, k \in \mathbb{N}^*$, and $x^{n, k} \to x^n$ as $k \to \infty$. Therefore,  the FCR property is not satisfied at any $x^n$.
Consequently, although the FCR property is satisfied at $\bar{x}$, it fails at every point of the sequence $\{x^n\}$. Since $x^n \to \bar{x}$ as $n \to \infty$, this demonstrates that the FCR property is not locally preserved at $\bar{x}$.
\endproof
\section{Characterization of CRCQ for  Affine Second-order Cone Constraints}
This section aims to establish a verifiable necessary and sufficient condition for the CRCQ to hold at a feasible point of the affine second-order cone constraint \eqref{SOCP}. To achieve this, we leverage the characterization of closed linear images of closed convex cones provided in \cite[Theorem 1]{Pataki07}. Utilizing this result alongside the geometry of the second-order cone, we derive the following characterization for the closedness of the set $H(\bar{x})$, as defined  in \eqref{setHx}.
\begin{Theorem}\label{cy1210}
	  Let $\bar{x}$ be a feasible point of \eqref{SOCP} and
 $H(\bar{x}):=\nabla g(\bar x)^*\left[N_{\mathcal{Q}_m}\big(g(\bar x)\big)\right].$
Then, the set $H(\bar{x})$ is closed if and only if one of the following conditions holds:
\begin{itemize}	
	\item[$(i)$] $g(\bar x)\in \mathcal{Q}_m\setminus\{0\}$;
	\item[$(ii)$] $g(\bar x)=0$ and $\text{\rm Im}(A)\cap \text{\rm int}(\mathcal{Q}_m)\neq\emptyset$;
	\item[$(iii)$] $g(\bar x)=0$ and ${\rm Im}(A)\cap \mathcal{Q}_m=\{0\}$;
	\item[$(iv)$] $g(\bar x)=0$ and $\text{\rm Im}(A)=\mathbb{R}v$ for some $v\in \text{\rm bd}^{+}(\mathcal{Q}_m)$.
	\end{itemize}
\end{Theorem}
\noindent {\it Proof}.  Since $\mathcal{Q}_m$ is reducible to $\mathcal{C}$ at $g(\bar{x})$ via the mapping $\Xi$, with $\mathcal{G} = \Xi \circ g$, it follows that $H(\bar{x}) = \nabla \mathcal{G}(\bar{x})^*(\mathcal{C}^*)$. We proceed by  examining the following cases.

 {\it Case 1:} $g(\bar x)\in {\rm int}(\mathcal{Q}_m)$. Then, we have $N_{\mathcal{Q}_m}\big(g(\bar{x})\big) = \{0\}$, which implies that
 $$H(\bar{x}) := \nabla g(\bar{x})^* [N_{\mathcal{Q}_m}\big(g(\bar{x})\big)] = \{0\}.$$ Consequently, $H(\bar{x})$ is trivially closed.

 {\it Case 2:} $g(\bar x)\in{\rm bd}^{+}(\mathcal{Q}_m)$. Then, the reduction mapping is given by $\Xi(y) = y_0 - \|y_r\|$ for $y = (y_0, y_r) \in \mathbb{R}^m$, with the reduced cone $\mathcal{C} = \mathbb{R}_+ $. In this case, since $H(\bar{x})$ is the image of the polyhedral cone $\mathcal{C}^*$ under the linear mapping $\nabla \mathcal{G}(\bar{x})^*$, it follows that $H(\bar{x})$ is polyhedral and, therefore, closed.

 {\it Case 3:} $g(\bar x)=0$. Let $\bar y \in \text{ri}(\text{Im}(A) \cap \mathcal{Q}_m)$ and let $F_0$ denote the minimal face of $\mathcal{Q}_m$ containing $\bar y$. Since $g(\bar x) = 0$, the reduction mapping $\Xi$ of $\mathcal{Q}_m$ at $g(\bar x)$ is  the identity function on $\R^m$  and $\mathcal{C} = \mathcal{Q}_m$. Recall that $\mathcal{Q}_m$ is a nice cone (see \cite[p. 399]{Pataki07}).  So, by \cite[Theorem 1.1]{Pataki07},  $H(\bar x)$ is closed if and only if
\begin{equation}\label{cy1210b}
A^*[\mathcal{Q}_m \cap F_0^{\perp}] = A^*[F_0^{\perp}].
\end{equation}
Since $\text{Im}(A)$ is a subspace and $0 \in \text{Im}(A) \cap \mathcal{Q}_m$, exactly one of the following possibilities must hold: $\text{Im}(A) \cap \text{int}(\mathcal{Q}_m) \neq \emptyset$, $\text{Im}(A) \cap \mathcal{Q}_m = \{0\}$,  or $\text{Im}(A) \cap \mathcal{Q}_m = \mathbb{R}_+ v$ for some $v \in \text{bd}^+(\mathcal{Q}_m)$. Furthermore, we see that
\[{\rm ri}({\rm Im}(A)\cap \mathcal{Q}_m)=\begin{cases}
\text{\rm Im}(A)\cap \text{\rm int}(\mathcal{Q}_m)  & \text{if} \quad  \text{\rm Im}(A)\cap \text{\rm int}(\mathcal{Q}_m)\neq\emptyset;\\
\{0\} & \text{if}  \quad {\rm Im}(A)\cap \mathcal{Q}_m=\{0\};\\
	\{tv|\; t>0\}  & \text{if}  \quad \text{\rm Im}(A)\cap \mathcal{Q}_m=\mathbb{R}_{+}v \;\text{for some}\; v\in {\rm bd}^{+}(\mathcal{Q}_m),
\end{cases}\]
and thus,
\[F_0=\begin{cases}
	\mathcal{Q}_m & \text{if} \quad  \text{\rm Im}(A)\cap \text{\rm int}(\mathcal{Q}_m)\neq\emptyset;\\
	\{0\} & \text{if}  \quad {\rm Im}(A)\cap \mathcal{Q}_m=\{0\};\\
	\{tv|\; t\geq 0\}  & \text{if}  \quad \text{\rm Im}(A)\cap \mathcal{Q}_m=\mathbb{R}_{+}v \;\text{for some}\; v\in {\rm bd}^{+}(\mathcal{Q}_m).
\end{cases}\]
\par\quad \quad{\it Case 3.1:} $g(\bar x)=0$ and $\text{\rm Im}(A)\cap \text{\rm int}(\mathcal{Q}_m)\neq\emptyset$. Then,  the minimal face is the entire cone, $F_0 = \mathcal{Q}_m$, which implies that $F_0^{\perp} = \{0\}$. It follows immediately that the equality \eqref{cy1210b} holds, and thus $H(\bar{x})$ is closed.
\par\quad \quad{\it Case 3.2:} $g(\bar x)=0$ and ${\rm Im}(A)\cap \mathcal{Q}_m=\{0\}.$
Then $F_0 =\{0\}$ and $\text{ri}({\rm Im}(A)\cap \mathcal{Q}_m)=\text{ri}(\{0\})=\{0\}.$   Hence, $0\in \text{ri}({\rm Im}(A)\cap \mathcal{Q}_m)$.
We observe that $\text{dir}(0, \text{Im}(A)) = \text{Im}(A)$ and $\text{dir}(0, \mathcal{Q}_m) = \mathcal{Q}_m$; furthermore, both $\text{Im}(A)$ and $\mathcal{Q}_m$ are closed.
This implies that  $$\text{dir}(0, \text{Im}(A)) \cap \text{dir}(0, \mathcal{Q}_m)=\text{Im}(A)\cap \mathcal{Q}_m= \text{cl}(\text{dir}(0, \text{Im}(A)))\cap \text{cl}(\text{dir}(0, \mathcal{Q}_m)).$$
 So, noting that $\text{Im}(A)$ and $\mathcal{Q}_m$ are nice cones, by \cite[Theorem 5.1]{Pataki07}, we see that
 $\mathcal{Q}_m^*+ \text{Im}(A)^*$ is closed.  Given that $\text{Im}(A)$ is a subspace, it follows that $\text{Im}(A)^* = \text{Im}(A)^\bot$. On the other hand, $\text{Im}(A)^\bot = \ker(A^*)$ by \cite[p. 155]{Luenb69}. Therefore, the sum  $\mathcal{Q}_m+ \ker(A^*)=-[\mathcal{Q}_m^*+ \text{Im}(A)^*]$ is closed.
Consequently, we have
$$\mathcal{Q}_m + \ker(A^*) = \text{cl}(-\mathcal{Q}_m^* - \text{Im}(A)^*) = -(\mathcal{Q}_m \cap \text{Im}(A))^* = \{0\}^* = \mathbb{R}^m,$$where the second equality follows from \cite[p. 398]{Pataki07}. This  implies that\begin{equation}\label{cy1310}A^*[\mathcal{Q}_m] = A^*[\mathbb{R}^m].\end{equation} Furthermore, since $F_0^{\perp} = \mathbb{R}^m$, \eqref{cy1310} is equivalent to \eqref{cy1210b}, justifying the closedness of $H(\bar{x})$.
\par\quad \quad{\it Case 3.3:} $g(\bar x)=0$ and $\text{\rm Im}(A)\cap \mathcal{Q}_m=\mathbb{R}_{+}v$ for some $v\in \text{\rm bd}^{+}(\mathcal{Q}_m)$.
 In this case, $F_0 = \{tv : t \geq 0\}$, which implies that $F_0^{\perp} = v^{\perp}$. Consequently, the condition \eqref{cy1210b} reduces to
\begin{equation}\label{cy1310c}
A^*(\mathcal{Q}_m \cap v^{\perp}) = A^*(v^{\perp}),
\end{equation}
which characterizes the closedness of $H(\bar{x})$.

We next prove that    \begin{equation}\label{cy1510}\mathcal{Q}_m \cap v^{\perp} = \{t(-v_0, v_r) \;|\; t \leq 0\},\end{equation}
 where $v = (v_0, v_r) \in \text{bd}^{+}(\mathcal{Q}_m)$.
 Take any  $u = (u_0, u_r) \in \mathcal{Q}_m \cap v^{\perp}$. Then   $u_0 v_0 + \langle u_r, v_r \rangle = 0$.
  This along with $v \in \text{bd}^{+}(\mathcal{Q}_m)$ shows that  $v_0 = \|v_r\| > 0$ and $\langle u_r, v_r \rangle = -u_0 \|v_r\|$. Hence,  using the Cauchy–Schwarz inequality, we have
  $$-u_0 \|v_r\| = \langle u_r, v_r \rangle \geq -\|u_r\|\|v_r\|,$$
   which yields $u_0 \leq \|u_r\|$ since $\|v_r\| > 0.$
    On the other hand, since $u \in \mathcal{Q}_m$, we get  $u_0 \geq \|u_r\|$.  Therefore, $u_0 = \|u_r\|$.
    Consequently,
    $$\langle u_r, v_r \rangle =-u_0v_0= -\|u_r\|\|v_r\|.$$
    This implies that  $u_r = t v_r$ for some $t \leq 0$.  Specifically, $t = -\|u_r\|/\|v_r\|$, from which it follows that $u_0 = \|u_r\| = -t\|v_r\| = -t v_0$. Thus, $u = t(-v_0, v_r)$ for some $t \leq 0$. This shows that
    \begin{equation}\label{cy1510a}\mathcal{Q}_m \cap v^{\perp} \subset \{t(-v_0, v_r) : t \leq 0\}.\end{equation}
Conversely, take $u=t(-v_0, v_r)$ for some  $t\leq 0$. Then, since $v_0=\|v_r\|$, we get
$$\la u, v\ra = t(-v_0^2+\|v_r\|^2)=0.$$ Furthermore, since $u_0=-tv_0=-t\|v_r\|=\|tv_r\|=\|u_r\|$, it follows that $u\in \text{bd}(\mathcal{Q}_m)\subset \mathcal{Q}_m$. Hence, $u\in \mathcal{Q}_m\cap v^{\perp}$. This shows that
\begin{equation}\label{cy1510b}\{t(-v_0, v_r) \;|\; t \leq 0\}\subset \mathcal{Q}_m \cap v^{\perp}.\end{equation}
From  \eqref{cy1510a} and \eqref{cy1510b} it follows that $\eqref{cy1510}$ holds.
So, the condition \eqref{cy1310c}
is equivalent to  \begin{equation}\label{cy1510d}\{ tA^*(-v_0, v_r)|\; t\leq 0\}=A^*(v^{\perp}).\end{equation}

Suppose that \eqref{cy1510d} holds. Then, since $A^*(v^{\perp})$ is a subspace, we get  $A^*(-v_0, v_r) = 0$, which yields $A^*(v^{\perp}) = \{0\}$.
 Thus, $v^{\perp} \subset \ker(A^*)$. Given that $v^{\perp}$ is a hyperplane and $\ker(A^*)$ is a subspace, we must have either $\ker(A^*) = \mathbb{R}^m$ or $\ker(A^*) = v^{\perp}$.  If $\ker(A^*) = \mathbb{R}^m$, then $\text{Im}(A) = \{0\}$. This is impossible since $\text{Im}(A) \cap \mathcal{Q}_m = \mathbb{R}_{+}v$. Consequently, we have $\ker(A^*) = v^{\perp}$. It then follows from \cite[Theorem 3, p. 157]{Luenb69} that
 $$\text{Im}(A) = (\ker(A^*))^{\perp} = v^{\perp\perp} = \text{span}(v) = \mathbb{R}v.$$
 Conversely, suppose that $\text{Im}(A)=\mathbb{R}v.$ Then, by \cite[Theorem 1, p. 155]{Luenb69},
 $$\ker(A^*)=\text{Im}(A)^\bot=v^\bot,$$
 that is,  $A^*( v^\bot)=\{0\}.$  On the other hand,
  $$\{0\} \subset \{ tA^*(-v_0, v_r)|\; t\leq 0\}= A^*(\mathcal{Q}_m \cap v^{\perp}) \subset A^*(v^{\perp}).$$
 Thus, \eqref{cy1510d} holds, and it follows that the condition $\text{Im}(A) = \mathbb{R}v$ is necessary and sufficient for \eqref{cy1510d} to be valid.
Therefore, in Case 3.3, $H(\bar{x})$ is closed if and only if $\text{Im}(A) = \mathbb{R}v$.
\endproof

As a direct consequence of Theorem~$\ref{cy1210}$, we obtain the following corollary characterizing the cases where $H(\bar{x})$ is not closed.
\begin{Corollary}\label{cy1610} Let $\bar{x}$ be a feasible point of \eqref{SOCP} and
 $H(\bar{x}):=\nabla g(\bar x)^*\left[N_{\mathcal{Q}_m}\big(g(\bar x)\big)\right].$
Then, the set $H(\bar{x})$ is not  closed if and only if
	\begin{equation}\label{cy1910}
		\begin{cases}
			\text{\rm Im}(A)\cap \mathcal{Q}_m=\mathbb{R}_{+}v, \\
			\text{\rm Im}(A)\neq \mathbb{R}v,
		\end{cases} \;\text{ for some } v\in \text{\rm bd}^{+}(\mathcal{Q}_m).
	\end{equation}
\end{Corollary}
Corollary \ref{cy1610} readily confirms the non-closedness of $H(\bar{x})$ in \cite[Example 1]{andreani23c}:
\begin{Example}{\rm  Consider the affine second-order cone constraint  as defined in \cite[Example 1]{andreani23c}:
 $$g(x) := Ax \in \mathcal{Q}_3,$$
 where  $Ax = (x_1, x_1, x_2)$ for $x=(x_1,x_2)\in \R^2$.  Let $\bar x= (0, 0)$ and $H(\bar x):=\nabla g(\bar x)^*\left[N_{\mathcal{Q}_3}\big(g(\bar x)\big)\right]$. Then $\bar x$ is a feasible point with $g(\bar x)=0.$  Furthermore,
 $$\text{\rm Im}(A)\cap \mathcal{Q}_3=\left\{(x_1, x_1, x_2)\ |\ (x_1,x_2)\in \R^2, x_1\geq \sqrt{x_1^2+x_2^2}\right\}=\R_+(1,1,0),$$
with $(1,1,0)\in {\rm bd}^+(\mathcal{Q}_3),$
while $$\text{\rm Im}(A)=\left\{(x_1, x_1, x_2)\ |\ (x_1,x_2)\in \R^2\right\}=\R(1,1)\times \R\not=\R(1,1,0).$$
Thus, $H(\bar{x})$ is not closed by Corollary \ref{cy1610}, in agreement with \cite[Example 1]{andreani23c}.
}\end{Example}
Combining Theorems \ref{cy209} and \ref{cy1210}, we obtain the following characterization of CRCQ for affine second-order cone constraints of the form:
\eqref{SOCP}.
\begin{Theorem}\label{cy2010}
Let $\bar{x}$ be a feasible point of \eqref{SOCP}. Then, the CRCQ is satisfied at $\bar x$ if and only if  one of the following conditions holds:
\begin{itemize}
   \item[$(i)$] $g(\bar{x}) \in {\rm int}(\mathcal{Q}_m)$;
    \item[$(ii)$] $g(\bar x)\in {\rm bd}^+(\mathcal{Q}_m)$ and the nondegeneracy condition holds at $\bar{x}$;
 \item[$(iii)$] $g(\bar x)\in {\rm bd}^+(\mathcal{Q}_m)$ and the reduced mapping $\mathcal{G}$ vanishes  on a neighborhood of $\bar{x}$;
\item[$(iv)$] $g(\bar x)=0$ and $\text{\rm Im}(A)\cap \text{\rm int}(\mathcal{Q}_m)\neq\emptyset$;
	\item[$(v)$] $g(\bar x)=0$ and ${\rm Im}(A)\cap \mathcal{Q}_m=\{0\}$;
	\item[$(vi)$] $g(\bar x)=0$ and $\text{\rm Im}(A)=\mathbb{R}v$ for some $v\in \text{\rm bd}^{+}(\mathcal{Q}_m)$.
  \end{itemize}
\end{Theorem}
\noindent {\it Proof}. To establish the result, we consider the following two cases:

{\it  Case 1:}  $g(\bar{x}) \in \mathcal{Q}_m \setminus \{0\}$. Then,  by Theorem \ref{cy1210}, the set $H(\bar{x})$ is closed. Consequently, CRCQ is equivalent to the FCR property. So, in this case, Theorem \ref{cy209} implies that CRCQ is satisfied at $\bar{x}$ if and only if one of the conditions $(i)$ through $(iii)$ holds.

{\it Case 2:} $g(\bar{x}) = 0$. Then, by Theorem \ref{cy209},  the FCR property holds  at $\bar{x}$. Therefore, the validity of CRCQ at $\bar{x}$ reduces to the closedness of $H(\bar{x})$. So, in this case, Theorem \ref{cy1210} implies that CRCQ is satisfied at $\bar{x}$ if and only if one of the conditions $(iv)$ through $(vi)$ holds.

Therefore, by combining these two cases, we get the desired conclusion.
\endproof

The following affine second-order cone constraint was provided in \cite[Example 4]{andreani23c} to demonstrate that CRCQ does not necessarily imply Seq-CRCQ. By applying Theorem~\ref{cy2010}, it can be readily verified that CRCQ holds in this specific instance.
\begin{Example}{\rm  Consider the affine second-order cone constraint  as defined in \cite[Example 4]{andreani23c}:
 $$g(x) := Ax \in \mathcal{Q}_3,$$
 where  $Ax = (x,-x, 0)$ for $x\in  \R$.  Let $\bar x= 0$.  Then $\bar x$ is a feasible point with $g(\bar x)=0.$  Furthermore,
 $$\text{\rm Im}(A)\cap \mathcal{Q}_3=\left\{(x, -x, 0)\ |\ x\in \R, x\geq \sqrt{(-x)^2+0^2}\right\}=\R_+(1,-1,0),$$
with $(1,-1,0)\in {\rm bd}^+(\mathcal{Q}_3).$
Therefore, by Theorem~\ref{cy2010},  CRCQ holds at $\bar x$, which is consistent with the result given in \cite[Example 4]{andreani23c}.
}\end{Example}

\section{The Equivalence of CRCQ and MSCQ for Affine Second-Order Cone Constraints}

In this section, we establish the equivalence between CRCQ and MSCQ for affine second-order cone constraints. The following theorem asserts this  relationship.
\begin{Theorem}\label{cy1511} Let $\bar{x}$ be a feasible point of \eqref{SOCP}. Then the following assertions are equivalent:
\begin{itemize}
\item[$(i)$]   CRCQ holds at $\bar x$;
 \item[$(ii)$]    MSCQ holds $\bar x$.
 \end{itemize}
\end{Theorem}
\noindent {\it Proof.} We first establish the implication $(i) \Rightarrow (ii)$. Assuming that CRCQ holds at $\bar{x}$, we consider the following three cases.

{\it Case 1.1:}  $g(\bar x)\in\text{\rm int}(\mathcal{Q}_m)$.  Then, $T_{\mathcal{Q}_m}\big(g(\bar{x})\big) = \mathbb{R}^m$, which implies $$\text{Im}\big(\nabla g(\bar{x})\big) + \text{lin}\left(T_{\mathcal{Q}_m}\big(g(\bar{x})\big)\right) = \mathbb{R}^m.$$ This equality ensures the nondegeneracy at $\bar{x}$, which in turn implies the satisfaction of MSCQ.

{\it Case 1.2}.  $g(\bar x)\in\text{\rm bd}^{+}(\mathcal{Q}_m)$:
Then, since CRCQ holds at $\bar{x}$, by Theorem $\ref{cy2010}$,  either the nondegeneracy condition is satisfied at $\bar{x}$ or $\mathcal{G}(x) = 0$ for all $x$ in a neighborhood of $\bar{x}$, where $\mathcal{G}(x):=g_0(x)-\|g_r(x)\|.$
If the nondegeneracy condition is satisfied at $\bar{x}$, then MSCQ also holds at this point.
 If $\mathcal{G}(x) = 0$ for all $x$ in a neighborhood of $\bar{x}$, then $\bar{x}$ is an interior point of the feasible set $\Omega := \{x \in \mathbb{R}^n \mid g(x) \in \mathcal{Q}_m\}$. Consequently,   for any $\kappa>0$,  we have
 $$\text{\rm dist}(x,\Omega)=0\leq \kappa \text{\rm dist}\big(g(x),\mathcal{Q}_m\big),$$ for every $x$ in a neighborhood of $\bar x$. Therefore, in this case, MSCQ holds at $\bar x.$

{\it Case 1.3:} $g(\bar x)=0$. Then, since CRCQ holds at $\bar{x}$, by Theorem \ref{cy2010},  one of the following conditions must be satisfied: $\text{Im}(A) \cap \text{int}(\mathcal{Q}_m) \neq \emptyset$, $\text{Im}(A) \cap \mathcal{Q}_m = \{0\}$, or $\text{Im}(A) = \mathbb{R}v$ for some $v \in \text{bd}^{+}(\mathcal{Q}_m)$.

\quad	{\it Case 1.3.1:} $g(\bar x)=0$ and $\text{\rm Im}(A)\cap\; \text{\rm int}(\mathcal{Q}_m)\neq\emptyset$. Then, there exists $d\in \mathbb{R}^n$ such that $Ad\in \text{\rm int}(\mathcal{Q}_m)$. Since $g(\bar{x}) = 0$, the reduced mapping is $\mathcal{G}(x) = (\Xi \circ g)(x) = g(x)$. It follows that$$g(\bar{x}) + \nabla g(\bar{x})d=\mathcal{G}(\bar{x}) + \nabla\mathcal{G}(\bar{x})d  = Ad \in \text{int}(\mathcal{Q}_m).$$ Thus, Robinson’s constraint qualification is satisfied at $\bar{x}$, which ensures the validity of  MSCQ.

\quad	{\it Case 1.3.2:} $g(\bar x)=0$ and  ${\rm Im}(A)\cap \mathcal{Q}_m=\{0\}$. Since $0=g(\bar x)=A\bar x+b$, we get $b=A(-\bar x)$ and $g(x)=Ax+b=A(x-\bar x)$. So, $g(x)\in \text{\rm Im}A$ and
\[\text{\rm dist}(g(x), \mathcal{Q}_m)=\text{\rm dist}\left( A(x-\bar x),\mathcal{Q}_m\right), \text{ for every } x\in\mathbb{R}^n.\]
If $A=0$ then  $\Omega=\R^n$ and for any $\kappa>0$,  we have
 $$\text{\rm dist}(x,\Omega)=0\leq \kappa \text{\rm dist}\big(g(x),\mathcal{Q}_m\big),$$ for every $x\in \R^n$. Therefore, in this case, MSCQ holds at $\bar x.$
Suppose now that $A \neq 0$. Then $S := \{y \in \text{Im}(A) \mid \|y\| = 1\}$ is a nonempty compact set, and   the function $y\mapsto \text{\rm dist}(y,  \mathcal{Q}_m)$ is continuous on  $S$. Furthermore, since ${\rm Im}(A)\cap \mathcal{Q}_m=\{0\}$, we have  $\text{\rm dist}(y,  \mathcal{Q}_m)>0$  for every $y\in S$. Therefore,
\[\eta:=\min_{y\in S}\text{\rm dist}(y,  \mathcal{Q}_m)>0,\]
which ensures that  $\text{\rm dist}(y,  \mathcal{Q}_m)\geq \eta>0$ for every $y\in S$.  Thus, since  $\mathcal{Q}_m$ is a cone and $\text{Im}(A)$ is a subspace, we get $\text{dist}(y, \mathcal{Q}_m) \geq \eta \|y\|$ for all $y \in \text{Im}(A)$.  Consequently,
\begin{equation}\label{cy811}
	\|A(x-\bar x)\|\leq \frac{1}{\eta}\text{\rm dist}\big(A(x-\bar x), \mathcal{Q}_m\big)=\frac{1}{\eta}\text{\rm dist}\big(g(x), \mathcal{Q}_m\big),
\end{equation}
 for every $x\in\mathbb{R}^n$.

 Consider now the mapping $h: \left( \text{\rm ker}(A)\right)^{\perp}\to \text{\rm Im}(A)$ defined by $h(v)=Av$ for all $v\in \left( \text{\rm ker}(A)\right)^{\perp}$.  Clearly, $h$ is linear. We next show that $h$ is a bijective mapping.
 Pick any $v\in\text{ker}(h)\subset \left( \text{\rm ker}(A)\right)^{\perp}$. Then $Av=h(v)=0$, which implies $v\in\text{ker}(A)$. Hence,  $v\in\text{ker}(A)\cap \left( \text{\rm ker}(A)\right)^{\perp}=\{0\}$. So, $v=0$ and $\text{ker}(h)=\{0\}$. This justifies the injectivity of  $h$.
To establish the surjectivity of $h$, consider an arbitrary $y \in \text{Im}(A)$. Then there exists $x \in \mathbb{R}^n$ such that $Ax = y$. Utilizing the orthogonal decomposition $\mathbb{R}^n = \ker(A) \oplus (\ker(A))^{\perp}$, we can write $x = k + v$ for some $k \in \ker(A)$ and $v \in (\ker(A))^{\perp}$. It follows that $h(v) = Av = A(x - k) = Ax = y$. This proves that  $h$ is surjective.
Consequently, the linear mapping $h$ is bijective, ensuring the existence of a linear inverse $h^{-1}: \text{Im}(A) \to (\ker(A))^{\perp}$. Since both $\text{Im}(A)$ and $(\ker(A))^{\perp}$ are finite-dimensional normed spaces, $h^{-1}$ is a continuous linear mapping. It follows that
\begin{equation}\label{cy811a}
\|h^{-1}(v)\| \leq M \|v\| \quad \text{for all } v \in \text{Im}(A),
\end{equation}
where $M := \|h^{-1}\|$ denotes the operator norm of the inverse.

Take any $x\in \R^n$. Since  $x-\bar x\in\mathbb{R}^n=\text{\rm ker}(A)\oplus \left( \text{\rm ker}(A)\right)^{\perp}$, we can write
$x-\bar x=k+v$ for  some  $k\in \text{\rm ker}(A) $ and   $v\in \left( \text{\rm ker}(A)\right)^{\perp}$. This yields $A\left( x-\bar x\right)=Ak+Av=Av=h(v)\in\text{\rm Im}(A) $. Set $x_{\Omega}:=\bar x+k$. Since $g(x_{\Omega}) = A(x_{\Omega}-\bar{x}) = Ak = 0\in \mathcal{Q}_m$, it follows that $x_{\Omega} \in \Omega$. Consequently, by appealing to  \eqref{cy811} and \eqref{cy811a}, we obtain
 \[\text{\rm dist}(x, \Omega)\leq \|x-x_{\Omega}\|=\|v\|=\|h^{-1}(A\left( x-\bar x\right))\|\leq M\|A\left( x-\bar x\right)\|\leq \frac{M}{\eta}\text{\rm dist}(g(x), \mathcal{Q}_m).\]
 This shows that MSCQ holds at $\bar x$.

\quad	{\it Case 1.3.3:}  $g(\bar x)=0$ and  $\text{\rm Im}(A)=\mathbb{R}v$ for some $v\in \text{\rm bd}^{+}(\mathcal{Q}_m)$.
Then there exists a nonzero continuous linear mapping $\alpha: \R^n\to \R$ such that $Ax=\alpha(x)v$ for every $x\in \R^n$.
This implies that $g(x)=Ax+b=\alpha(x)v+b$. Since $g(\bar x)=0$, we obtain $g(x)=(\alpha(x)-\alpha(\bar x))v=t(x)v$, where $t(x):=\alpha(x)-\alpha(\bar x)$. Observing that $v_0 = \|v_r\| > 0$, it follows that $g(x) = \big(t(x)v_0, t(x)v_r\big) \in \mathcal{Q}_m$ if and only if $t(x) \geq 0$. Hence,
\[\Omega:=\{x\in\mathbb{R}^n|\; g(x)\in \mathcal{Q}_m\}=\{x\in\mathbb{R}^n|\; t(x)\geq 0\}.\]
Since $\alpha: \R^n\to \R$  is a nonzero continuous linear mapping, by the Riesz representation theorem, there exists a nonzero vector $a\in\R^n$ such that $\alpha(x)=\langle x,a \rangle $ for every $x\in \R^n$.
Note that $t(x)=\langle x-\bar x, a\rangle$. Pick any $x\in \R^n$ with $t(x)<0$. Then, the projection $p:=\Pi_{\Omega}(x)$ of $x$ onto $\Omega$  lies on the line through $x$ that is perpendicular to the hyperplane $\{u\in \R^n\ |\ t(u)=0\}$. So, we can  find $s\in\mathbb{R}$ such that $p-x=sa$ and $t(p)=0$. We have
\[0=t(p)=\langle p-\bar x, a\rangle=\la a, x+sa-\bar x\ra=\alpha(x)- \alpha(\bar x)+s\|a\|^2=t(x)+s\|a\|^2,\]
which implies that $s=-\frac{t(x)}{\|a\|^2}$ and $p=x+sa=x-\frac{t(x)}{\|a\|^2}a$. Hence,
\[p=\Pi_{\Omega}(x)=\begin{cases}
	x & \text{ if } t(x)\geq 0, \\
	x-\frac{t(x)}{\|a\|^2}a & \text{ if } t(x)< 0.
\end{cases}\]
Consequently,
\[\text{\rm dist}(x, \Omega)=\|p-x\|=\begin{cases}
	0 & \text{ if } t(x)\geq 0, \\
	\frac{|t(x)|}{\|a\|} & \text{ if } t(x)< 0.
\end{cases}\]
Moreover, since $g(x)=t(x)v$, by  $\eqref{h312}$, we get
\[\text{\rm dist}(g(x), \mathcal{Q}_m)=\begin{cases}
		0 & \text{ if } t(x)\geq 0, \\
		|t(x)|\|v\| & \text{ if } t(x)< 0.
\end{cases}\]
Pick now any $x\in\mathbb{R}^n$. If $t(x)\geq 0$, then $\text{\rm dist}(x, \Omega)=0=\text{\rm dist}(g(x), \mathcal{Q}_m)$. Otherwise, if $t(x)<0$, then
\begin{eqnarray*}
	\text{\rm dist}(x, \Omega)=\frac{|t(x)|}{\|a\|}=\kappa \text{\rm dist}(g(x), \mathcal{Q}_m),
\end{eqnarray*}
where $\kappa:=\frac{1}{\|a\|\|v\|}>0$. Hence, MSCQ holds at $\bar x$.

We now establish the implication $(i) \Rightarrow (ii)$. Suppose that MSCQ holds at $\bar{x}$. If $g(\bar x)=0$ then  $H(\bar{x})= N_{\Omega}(\bar{x})$ is closed, and by Theorem \ref{cy209}, the FCR is satisfied. Thus, in this case,  CRCQ holds at $\bar{x}$. If either $g(\bar{x}) \in \text{int}(\mathcal{Q}_m)$ or $g(\bar{x}) \in \text{bd}^+(\mathcal{Q}_m)$ with the nondegeneracy condition holding at $\bar{x}$, then Theorem \ref{cy2010} ensures the validity of CRCQ at $\bar{x}$. It remains only to examine the case where $g(\bar{x}) \in \text{bd}^+(\mathcal{Q}_m)$ and the nondegeneracy condition is violated at $\bar{x}$. In this case, the reduced mapping defined on a neighborhood $\mathcal{U}$ of $\bar x$ by
$\mathcal{G}(x)=\phi(x):=g_0(x)-\|g_r(x)\|$.
  The failure of the nondegeneracy condition at $\bar{x}$ implies that $\nabla \mathcal{G}(\bar{x}) =\nabla\phi(\bar x)= 0$. On the other hand, since MSCQ is satisfied at $\bar{x}$, we have $$N_\Omega(\bar{x}) = \nabla \mathcal{G}(\bar{x})^* N_{\mathcal{Q}_m}(g(\bar{x})).$$ Therefore,  $N_\Omega(\bar{x}) = \{0\}$, which implies  that $\bar{x} \in \text{int}(\Omega)$; see \cite[Corollary 2.24]{Mord06}.
Let $A$ be partitioned as $A = \begin{bmatrix} A_0 \\ A_r \end{bmatrix}$, where $A_0$ denotes the first row and $A_r$ is the submatrix comprising the remaining rows. Then $g(x)=Ax+b=(A_0x+b_0, A_rx+b_r)$ for all $x\in \R^n.$
 Thus $g_0(x)=A_0x+b_0$ and $g_r(x)=A_rx+b_r$  for all $x\in \R^n$. Consequently, $\nabla g_0(x)=A_0$ and $\nabla g_r(x)=A_r$.
  Hence
  $$\nabla \mathcal{G}(x)=\nabla\phi(x)=\nabla g_0(x)-\frac{1}{\|g_r(x)\|}g_r(x)^*\nabla g_r(x)=A_0-\frac{1}{\|g_r(x)\|}g_r(x)^*A_r.$$
Furthermore,
\[\nabla\phi^2(x)=\nabla^2g_0(x)-\frac{g_r(x)^*}{\|g_r(x)\|}\nabla^2g_r(x)-\frac{1}{\|g_r(x)\|}\nabla g_r(x)^*\left(I-\frac{g_r(x)g_r(x)^*}{\|g_r(x)\|^2} \right)\nabla g_r(x), \]
(see \cite[p.~3123]{BGM19}). So, we get
$$\nabla\phi^2(\bar x)=-\frac{1}{\|g_r(\bar x)\|}A_r^*\left(I-\frac{g_r(\bar x)g_r(\bar x)^*}{\|g_r(\bar x)\|^2}\right)A_r.$$  Note that  $I-\frac{g_r(\bar x)g_r(\bar x)^*}{\|g_r(\bar x)\|^2}$ is  positive semidefinite.
Consequently, $\nabla^2\phi(\bar x)$ is negative semidefinite.
We next aim to prove that $\nabla^2\phi(\bar x)=0$. Suppose, for the sake of contradiction, that $\nabla^2\phi(\bar x) \neq 0$. Then there exists  $d \in \mathbb{R}^n$ such that $\langle d, \nabla^2\phi(\bar x)d \rangle < 0$. Given that $\phi(\bar x) = 0$ and $\nabla \phi(\bar x) = 0$, the Taylor expansion of $\phi$ around $\bar x$ yields that  $$ \phi(\bar x+td) = \phi(\bar x) + t\langle \nabla \phi(\bar x), d \rangle + \frac{t^2}{2} \langle d, \nabla^2\phi(\bar x)d \rangle + o(t^2) = \frac{t^2}{2} \langle d, \nabla^2\phi(\bar x)d \rangle + o(t^2). $$ Consequently, for all non-zero $t$ with $|t|$ sufficiently small, we have $$ \phi(\bar x+td) < \phi(\bar x) = 0, $$ which provides the desired contradiction since  $\bar{x} \in \text{int}(\Omega)$ and $\Omega=\{x\in\mathbb{R}^n|\;\phi(x)\geq 0\}$.
  This shows that $\nabla^2\phi(\bar x)=0$.
  In other words, $A_r^*\left(I-\frac{g_r(\bar x)g_r(\bar x)^*}{\|g_r(\bar x)\|^2}\right)A_r=0$. Thus for all $x\in \mathbb{R}^n$, we get
  $$0=(A_rx)^*\left(I-\frac{g_r(\bar x)g_r(\bar x)^*}{\|g_r(\bar x)\|^2}\right)A_rx,$$ or equivalently, $\|A_rx\|^2=\frac{1}{\|g_r(\bar x)\|^2}\big(g_r(\bar x)^*A_rx\big)^2$. Moreover, by the Cauchy-Schwarz inequality,
\[\|A_rx\|^2=\frac{1}{\|g_r(\bar x)\|^2}\left(g_r(\bar x)^*A_rx\right)^2\leq \frac{1}{\|g_r(\bar x)\|^2}\|g_r(\bar x)\|^2\|A_rx\|^2=\|A_rx\|^2.\]
So, for every $x\in \mathbb{R}^n$ there exists $\lambda_x\in\mathbb{R}$ such that $A_rx=\lambda_xu$, where $u:=\frac{1}{\|g_r(\bar x)\|}g_r(\bar x).$
  Due to the linearity of $A_r$ and $u\not=0$, the mapping $\lambda: \mathbb{R}^n\to \mathbb{R}$ defined by $\lambda(x)=\lambda_x$ is a linear functional. Let $\{e_1, \ldots, e_n\}$ be the canonical basis of $\mathbb{R}^n$. Then, for every $x=(x_1, \ldots, x_n)\in \mathbb{R}^n$, we see that  \[\lambda(x)=\lambda\left( \sum\limits_{i=1}^n x_ie_i\right)= \sum\limits_{i=1}^n x_i\lambda(e_i)=\langle w, x\rangle,\]  where $w:=\big(\lambda(e_1), \ldots, \lambda(e_n)\big)\in\mathbb{R}^n$. Hence, for every $x\in \mathbb{R}^n$,
  $$A_rx= \langle w, x\rangle u=(uw^*)x.$$ This implies  that $A_r=uw^*$. On the other hand, since $\nabla \phi(\bar x)=0$, we have  $A_0=u^*A_r.$
  Therefore,
  $$A_0 =u^*(uw^*)=(u^*u)w^*=\|u\|^2w^*=w^*.$$ So,
\begin{eqnarray*}
	g_0(x)=A_0x+b_0=w^*x+b_0, \; g_r(x)=A_rx+b_r=\langle w,x\rangle u+b_r=(uw^*)x+b_r=u(w^*x)+b_r.
\end{eqnarray*}
This along with  $g_0(\bar x)=\|g_r(\bar x)\|$ and $u=\frac{g_r(\bar x)}{\|g_r(x)\|}$ gives us that
 $$g_r(\bar x)=\|g_r(\bar x)\|u=g_0(\bar x)u=(w^*\bar x+b_0)u=w^*\bar xu+b_0u=\langle w,\bar x\rangle u+b_0u.$$ Hence $b_r=b_0u$. On the other hand, $b_r=g_r(\bar x)- u(w^*\bar x)=(\|g_r(\bar x)\|-w^*\bar x)u$. Let $c:=\|g_r(\bar x)\|-w^*\bar x\in\mathbb{R}$. Then, $b_r=cu$, and $b_0=c$. Therefore,
\[	g_0(x)=w^*x+c, \; g_r(x)=u(w^*x)+cu=(w^*x+c)u.\]
Consequently, $\phi(x)=g_0(x)-\|g_r(x)\|=w^*x+c-|w^*x+c|\|u\|=w^*x+c-|w^*x+c|$.
Now observing that $g_0(\bar x)=w^*\bar x+c>0$ and $g_0$ is continuous, there exists a neighborhood $V$ of $\bar x$ such that $g_0(x)=w^*x+c>0$ for all $x\in V$. So, $\mathcal{G}(x)=\phi(x)=w^*x+c-(w^*x+c)=0$ for all $x\in V$. Therefore, by Theorem \ref{cy2010}, CRCQ holds at $\bar x$.
The proof is completed.
\endproof
\section{Concluding Remarks}

This paper has investigated CRCQ within the framework of linear nonpolyhedral second-order cone programs (SOCPs). We first demonstrated that the facial constant rank property, which is a key requirement for the validity of  CRCQ,  is not universally satisfied in this setting. By deriving a necessary and sufficient condition for this property to hold, we established an easily verifiable characterization of CRCQ. Finally, utilizing this characterization, we proved the equivalence between CRCQ and MSCQ in the linear SOCP context.  Future research could explore several promising directions. The priority is determining whether these results extend to other linear nonpolyhedral cone programming settings, including semidefinite and reducible cone programs. Furthermore, identifying specific  practical problems that satisfy the CRCQ within the nonpolyhedral cone programming framework would be of significant value.
 

\end{document}